\documentclass{article}

\usepackage[a4paper, total={6in, 8in}]{geometry}

\usepackage{url}
\usepackage{
amsmath,
amsthm,
amscd,
amssymb,
}
\usepackage{epigraph}
\usepackage{comment}

\theoremstyle{plain}
\newtheorem{lemma}{Lemma}
\newtheorem{definition}{Definition}
\newtheorem{corollary}{Corollary}
\newtheorem{proposition}{Proposition}
\newtheorem{theorem}{Theorem}

\newtheorem{question}{Question}

\newtheoremstyle{derp}
{3pt}
{3pt}
{}
{}
{\upshape}
{:}
{.5em}
{}
\theoremstyle{derp}

\newcommand{\Z}{\mathbb{Z}}
\newcommand{\N}{\mathbb{N}}

\title{Trees in positive entropy subshifts}

\author{
Ville Salo\thanks{This research was funded by Academy of Finland grant number 2608073211.} \\
vosalo@utu.fi
}

\begin{document}
\maketitle

\begin{abstract}
I give a simple proof for the fact that positive entropy subshifts contain infinite binary trees where branching happens synchronously in each branch, and that the branching times form a set with positive lower asymptotic density.
\end{abstract}

\epigraph{The proof is trivial! Just view it as a \\
\textbf{rational metric space} \\
whose elements are \\
\textbf{countable combinatorial games}}{\url{www.theproofistrivial.com/} \\ (generated during CANT 2016)}

\section{Introduction}

In topological dynamics, the topological entropy of a dynamical system measures the information in orbits, by counting the exponential growth rate of different partial orbits with a certain of accuracy. One example of a dynamical system is a \emph{subshift}, a dynamical system whose points are infinite words over a finite alphabet $A$, and the dynamics is the left shift. One-way entropy can arise in a subshift when, somewhere in partial orbits, we literally see all distinct words $A^n$ up to $n$. One may ask to what extent entropy always arises from such ``free choices''.

Two formal notions that most directly capture this idea are the \emph{independence entropy} of a subshift and \emph{independence sets}. The former measures how much entropy comes from ``set-theoretic hyperrectangles'', namely sets of the form $A_1 \times A_2 \times \cdots A_n$ which are contained in the language of the subshift as $n \rightarrow \infty$, where $A_i \subseteq A$ are any subsets of the alphabet and we consider consecutive letters only, see \cite{LoMaPa13}, (p.~303) for the precise definition. Positive entropy does not imply positive independence entropy (see Proposition~\ref{prop:NoIndependenceEntropy}).

Independence sets, on the other hand, are defined similarly in \cite{KuKwLi14a,FaKuKwLi15}, but we allow gaps between the positions where the prescribed letters appear, and there is no constraint on the words seen in the gaps between the positions. More precisely, in the case of a subshift $X$ with a binary alphabet $A = \{0,1\}$, an \emph{independence set} is a set $N \subseteq \N$ such that $X|_N = \{0,1\}^N$. It turns out that the existence of an independence set of positive lower asymptotic density does capture positive entropy, see \cite{KuKwLi14a,FaKuKwLi15} for a proof in this binary subshift case, and \cite{GlWe95,HuYe06,KeLi07} for analogous results in more general settings.

The following theorem gives a new proof of this result (see Corollary~\ref{cor:IndependenceSet}), but it is stronger, as there is an additional constraint on how the independent choices must be realized. It also applies in the case of a general alphabet.

\begin{theorem} 
\label{thm:Main}
A subshift $X \subseteq A^\N$ has positive entropy if and only if it contains a steadily branching binary tree.
\end{theorem}

We can interpret $A^\N$ as the boundary of an $|A|$-ary tree and $X$ a closed subset of this boundary. A closed subset of the boundary of a tree is determined by finite subpaths in the tree, and it is in this sense that $X$ contains the binary tree. Steadily branching means that the embedded binary tree branches at a uniform sequence of times which form a set of positive lower asymptotic density. A formal definition is given in Section~\ref{sec:Winning}.

Positive entropy refers to topological entropy, and in the context of subshifts, means the exponential growth of the number of words. Countable subshifts can have a growth rate arbitrarily close to exponential (see Proposition~\ref{prop:Countable}), but clearly contain trees with only finitely many branchings. Thus, if we weaken the assumption of positive entropy, the existence of a steadily branching binary tree can fail in a very strong sense.

It is also easy to find examples with positive entropy where it is not possible to find a tree where branchings happen a syndetic set of times (see Proposition~\ref{prop:NoSyndetic}). There are also examples where no individual tree which branches at a uniform sequence of times ``captures'' all of the entropy, as can be seen from \cite{LoMaPa13}, (Example~2). These remarks show that we cannot expect to essentially improve the conclusion about the tree under the assumption of positive entropy.

The proof of Theorem~\ref{thm:Main} is straightforward using the standard results about the winning shift, which is the set of winning turn orders in a certain word-building game associated to $X$. The winning shift is known to have the same entropy as $X$. This implies it has a point with positive density, and this point is the branching structure of a steadily branching tree in~$X$. This note is simply a self-contained elaboration of this deduction. I give a combinatorial proof directly from the definitions, and also include an ergodic theoretic proof.

It is easy to deduce a two-sided version of the result as well. For $x \in A^{-\N}$ and $y \in A^\N$, write $x \cdot y \in A^\Z$ for the configuration with $x$ and $y$ back to back.

\begin{theorem}
\label{thm:Z}
A subshift $X \subseteq A^\Z$ has positive entropy if and only if for some $x \in A^{-\N}$, the set $\{y \in A^\N \;:\; x \cdot y \in X \}$ contains a steadily branching binary tree.
\end{theorem}

The winning shift was introduced in \cite{SaTo14c}, and has also been studied in \cite{PeSa19}. We found out while working on \cite{PeSa19} that a concept equivalent to the winning shift was discovered in the set systems setting already in 2002 \cite{AnRoSa02} (even if ostensibly only for binary words), and I found out while working on the present paper that \cite{AnRoSa02} has been applied by \mbox{dynamicists \cite{KuKwLi14a,FaKuKwLi15}} (this paper is independent from \cite{SaTo14c}, but almost contemporary), and their proof that this gives large independence sets essentially boils down to the same proof as I give here. Nevertheless, I feel that though Theorem~\ref{thm:Main} is an interesting statement about subshifts, to our knowledge the statement does not appear in the literature, and its proof through winning shifts is worth making explicit. Furthermore, unlike Theorem~\ref{thm:Main}, the statement about independence sets does not trivially generalize to non-binary alphabets.

As discussed, I obtain a new proof of the fact entropy implies the existence of certain types of independence sets in the case of a binary alphabet. 

\begin{corollary}
\label{cor:IndependenceSet}
Let $X \subseteq \{0,1\}^\N$ be a subshift. Then, $X$ has positive entropy if and only if it has an independence set of positive lower asymptotic density.
\end{corollary}

Again, I mention that this result is proved in \cite{KuKwLi14a,FaKuKwLi15}, and that generalizations appear in \cite{GlWe95,HuYe06,KeLi07}. To our knowledge, the first reference where the Sauer-Shelah lemma appears in a dynamics context is \cite{GlWe95}. The most naive generalization of Corollary~\ref{cor:IndependenceSet}, the existence of a high density set $N$ with $X|_N = A^N$, fails for a non-binary alphabet $A$ (even for trivial reasons, by artificially increasing the alphabet of the subshift). However, one can state an equivalent condition for having higher entropy in this way \cite{Be00}, (Theorem~8.3), \cite{FaKuKwLi15}, (Theorem~2). (The non-trivial direction of this equivalence also appears in \cite{SaTo14c}, (Proposition~5.9).)

As I show in this note, winning shifts say something stronger than Corollary~\ref{cor:IndependenceSet} in the case of one-dimensional subshifts, but it is not clear how to generalize them to other settings, at least in a way that allows seeing them naturally as a dynamical system.

\begin{question}
Does every expansive system of positive entropy contain a steadily branching binary tree (and what does that even mean)? More generally, can the notion of winning shift be extended to general (expansive) systems? Group/monoid actions? 
\end{question}


\section{Definitions}

Let $\N \ni 0$ be the natural numbers. An \emph{alphabet} is a finite set $A$. \emph{Words} are elements of the free monoid generated by $A$, denoted by $A^*$. Elements of $A$ are seen as belonging to $A^*$ when convenient, $u \cdot v$ denotes the multiplication in the free monoid (formal concatenation). The \emph{length} $|w|$ of a word is defined in the obvious way. Write $A^n = \{w \in A^* \;:\; |w| = n\}$. 

The set $A^\N$, for a finite alphabet $A$ (or later $A^\Z$), always carries the product topology, and it is homeomorphic to Cantor space if $|A| \geq 2$. A \emph{subshift} is a closed subset of $A^\N$ for a finite alphabet $A$ (or later $A^\Z$) satisfying $\sigma(X) \subseteq X$ where $\sigma$ is the \emph{shift} $\sigma(x)_i = x_{i+1}$. The \emph{words of a subshift} $X \subseteq A^\N$ are the words $w \in A^*$ such that $x|_{[0,|w|-1]} = w$ for some $x \in X$. The words of a subshift form its \emph{language}. Writing $L_n$ for the intersection of the language of $X$ with $A^n$, the \emph{(topological) entropy} of a subshift $X$ is the exponential growth rate of the number of words, $h(X) = \lim_{n \rightarrow \infty} \frac{\log |L_n|}{n}$ (which exists by the subadditivity of $\log |L_n|$). The \emph{lower asymptotic density} of $N \subseteq \N$ is $\liminf_k \frac{|N \cap [0,k-1]|}{k}$.

Two subshifts $X \subset A^\N, Y \subset B^\N$ are \emph{topologically conjugate} if there is a shift-commuting homeomorphism $\phi : X \to Y$. A subshift is \emph{countable} if it countably has many points. A set $N \subseteq \N$ is \emph{syndetic} if $\N \subseteq N+[-n,n]$ where $+$ denotes the Minkowski sum.

\section{The Winning Shift}
\label{sec:Winning}

For $z, b \in \N^\N$, denote $z \leq b \iff \forall i: z_i \leq b_i$.

\begin{definition}
Let $A$ be a finite alphabet and let $X \subseteq A^\N$. A \emph{tree in $X$ with branching structure $b \in \N^\N$} is a set of sequences $(x^z)_{\N^\N \ni z \leq b}$ where each $x^z$ is an element of $X$, and:
\[ x^z_{[0,i)} = x^{z'}_{[0,i)} \iff z_{[0,i)} = {z'}_{[0,i)}. \]
A \emph{steadily branching binary tree} is a tree with branching structure $b \in \{0,1\}^\N$ satisfying $\liminf_{i \rightarrow \infty} \frac{\sum_{j = 0}^{i-1} b_j}{i} > 0$.
\end{definition}

In other words, for distinct sequences $z, z' \in \N^\N$, the first position where $x^z$ and $x^{z'}$ differ is the same as the first position where $z$ and $z'$ differ. For binary $b$, this means that the nonzero positions in $b$ are the positions where our tree must branch in two, and more generally, $b_i = n$ means the tree must branch $n$ times in position $i$, explaining why we call this sequence a branching structure. Note also that if $X \subseteq A^\N$ and $X$ contain a tree with branching structure $b \in \N^\N$, then necessarily $b_i \leq |A|-1$ for all $i \in \N$.

\begin{definition}
Let $X \subseteq A^\N$ be a subshift. Let $W(X)$, the \emph{winning shift} of $X$, be the set of all branching structures of trees in $X$.
\end{definition}

This is defined in a game-theoretic framework in \cite{SaTo14c}: a tree in $X$ with branching structure $z$ can be interpreted as a winning strategy for the first player in a word-building game where on the $i$th turn, the first player picks a set of $z_i+1$ symbols, then the second picks one of them, and the first player wins if the word obtained in the limit is in~$X$. 

It is shown in \cite{SaTo14c} that the winning shift is indeed a subshift. A subshift $Y \subseteq \N^\N$ is \emph{hereditary} \cite{KeLi07} if for all $y' \in \N^\N$ we have:
\[ (y \in Y \wedge \forall i: y'_i \leq y_i) \implies y' \in Y. \]

The following fact is Proposition 5.7 of \cite{SaTo14c} (the hereditarity claim is trivial). The notation differs slightly, in the reference $\tilde W(X)$ is used for what we call $W(X)$, and their $W(X)$ is the same as ours in the binary case, but differs in general.

\begin{lemma}
The subshift $W(X)$ is hereditary and has the same number of words of each length as~$X$.
\end{lemma}

\begin{proof}[Proof sketch]
The hereditarity claim is trivial. For the claim about words, one can define $W(L)$ for $L \subseteq A^n$ analogously to the subshift case (using finite trees). Letting $L \subseteq A^n$ be the language of $X$ intersected with $A^n$, $W(L)$ is the language of $W(X)$ intersected with $A^n$. Now, we write $L = 0 L_0 \sqcup 1 L_1 \sqcup \cdots \sqcup (k-1) L_{k-1}$ and exchange a suitable sum:
\begin{align*}
|W(L)| &= \left|\{ i \cdot w \;:\; |\{j \;:\; w \in W(L_j)\}| > i\}\right| \\
&= \sum_i \sum_{{w \in A^{n-1}}\atop{|\{j \;:\; w \in W(L_j)\}| > i}} 1 \\
&= \sum_{w \in A^{n-1}} \sum_{i = 0}^{|\{j \;:\; w \in W(L_j)\}|-1} 1 \\
&= \sum_{w \in A^{n-1}} |\{j \;:\; w \in W(L_j)\}| \\
&= \sum_j |W(L_j)| = \sum_j |L_j| = |L|.
\end{align*}
The first equality is true by definition, and the penultimate one by induction.
\end{proof}

For a finite word $w \in \N^*$, write $\sum w = \sum_i w_i$ for the sum of the symbols in $w$, and $|w|$ for the length of $w$ as a word. The key to finding steadily branching trees is to study the \emph{density} $\sum w/|w|$ of a word $w$. If $Y \subseteq \{0,1\}^\N$ is a subshift, by $Y^k$ we mean the Cartesian power with the diagonal action, interpreted in an obvious way as an alphabet over the alphabet $\{0,1\}^k$.

We observe a simple combinatorial lemma.

\begin{lemma}
\label{lem:Density}
If a subshift $Y$ has positive entropy and is hereditary, then for some $\beta > 0$ there exist arbitrarily long words $w$ of $Y \cap \{0,1\}^\N$ with $\sum w/|w| \geq \beta$.
\end{lemma}

\begin{proof}
Let $\{0,1,...,k-1\}$ be the alphabet of $Y$. Let:
\[ b(Y) = \{x \in \{0,1\}^\N \;:\; \exists y \in Y: \forall i: y_i \neq 0 \iff x_i \neq 0 \}. \]
Since $Y$ is hereditary, it is easy to see that the number of words in $Y$ of length $n$ is at most the number of words in $b(Y)^k$ of length $n$. Thus, if $Y$ has positive entropy, so does $b(Y)^k$, and thus so does $b(Y)$. Suppose thus that the number of words in $b(Y)$ of every length $n$ is at least $\alpha^n$ for some $\alpha > 0$, as is clearly implied by positive entropy.

The number of words of length $n$ with at most $k$ many $1$s is at most:
\[ \binom{n}{k} 2^k < \left(\frac{n \cdot e}{k}\right)^k 2^k. \]

Setting $k = \beta n$, this becomes $((\frac{n \cdot e}{\beta n})^\beta 2^\beta)^n$ and we observe that
as $\beta \rightarrow 0$, we have $2^\beta \rightarrow 1$ and $(\frac{n \cdot e}{\beta n})^\beta = (\frac{e}{\beta})^\beta \rightarrow 1$, so for small enough $\beta > 0$, we have $(\frac{n \cdot e}{\beta n})^\beta 2^\beta < \alpha$, and thus there must be words of length $n$ in $b(Y)$ with at least $\beta n$ many $1$s, for arbitrarily large $n$. These are words of $Y \cap \{0,1\}^\N$ since $Y$ is hereditary.
\end{proof}

This argument appears also in \cite{Kw13} (Theorem~4.8) (this theorem already includes the statement about lower asymptotic density, which we deduce in the next section).

\section{The Proofs}

\begin{proof}[Proof of Theorem~\ref{thm:Main}]
Obviously, a steadily branching tree implies positive entropy.

For the other direction, let $A \subseteq \N$ be a finite alphabet and $Y \subseteq A^\N$ a subshift. Write $s_n$ for the maximal sum $\sum w$ of a word $w$ of length $n$ in $Y$. This sequence is clearly subadditive, so $\lim_{n \rightarrow \infty} s_n/n$ exists, say $\lim_n s_n/n = \alpha$.

We outline the usual addendum that there must be a configuration $y \in Y$ such that every prefix $w$ of $y$ satisfies $\sum w/|w| \geq \alpha$. Supposing that this is not the case, then every point $X$ has a prefix $w$ such that $\sum w/|w| < \alpha$, and this must happen after a bounded number of steps by compactness, thus there exists $\epsilon > 0$ such that for some $m$, we always find a prefix $w$ of length at most $m$ in any point $y \in Y$, such that $\sum w/|w| < \alpha - \epsilon$. Now, given any long word $w$ of $Y$, we can split it as $w = w^0 w^1 \cdots w^k$ (where $i$ in $w^i$ denotes a superscript, not a power) with $|w^k| \leq m$ and $\sum w^i/|w^i| < \alpha - \epsilon$ for all $i < k$. If $w$ is long enough, then since $\sum uv/|uv| \leq \max(\sum u/|u|, \sum v/|v|)$ for all words $u, v$, we have:
\[ \sum w/|w| < \alpha - \epsilon/2 \]
for all long enough words, a contradiction to $\lim_n s_n/n = \alpha$.

Now, we set $Y = W(X)$ and apply Lemma~\ref{lem:Density}. The lemma implies that $\alpha > 0$ in the above. The point $y \in W(X)$ whose density stays above $\alpha$ gives the branching structure of a steadily branching tree in $X$.
\end{proof}

The ``usual addendum'' is well known. The author learned it from \cite{KaKe97,Je14} and P.\ Guillon; Jeandel attributes it to \cite{Ma99,FeHu10}, and it also appears in \cite{Kw13,FaKuKwLi15} where it is attributed to \cite{Pe88,Ru78}.

The point of this note was to provide a combinatorial proof of Theorem~\ref{thm:Main} from the first principles, as this is easy to do. However, I have included also the proof using ergodic theory (see \cite{DeGrSi76} for a basic reference).

\begin{proof}[Ergodic theoretic proof of Theorem~\ref{thm:Main}]
Obviously a steadily branching tree implies positive entropy. For the other direction, since the winning shift has positive entropy, it admits an invariant measure $\mu$ with positive entropy, thus $\mu([a]) > 0$ for some $a > 0$. By the ergodic decomposition there exists such an ergodic measure, and by the pointwise ergodic theorem, there is a point $y \in W(X)$ where the lower asymptotic density of $a$s is positive, giving the result.
\end{proof}

We conclude with the proofs of Theorem~\ref{thm:Z} and Corollary~\ref{cor:IndependenceSet}.

\begin{proof}[Proof of Theorem~\ref{thm:Z}]
Let $X_R \subseteq A^\N$ be the subshift of the right tails of points in $X$ and apply Theorem~\ref{thm:Main}. Let $y \in \{0,1\}^\N$ be the branching structure of some steadily branching tree and $\alpha$ the lower asymptotic density. It is easy to see that for any $n$, we can find a finite prefix $w$ of $y = wx$ of a length of at least $n$ such that for all prefixes $u$ of $x$ of length at most $n$, $\sum u/|u| \geq \alpha$. Otherwise, by cutting long words into ones of length at most $n$, as in the previous argument, we see that the lower asymptotic density is less than $\alpha$.

The fact we can find such $wx$ as a branching structure implies that the tree corresponding to $x$ can follow at least one word of length $|w|$. Letting $n$ tend to infinity, by compactness, we obtain a steadily branching tree that can follow some left tail.
\end{proof}

\begin{proof}[Proof of Corollary~\ref{cor:IndependenceSet}]
Supposing $X$ has positive entropy, let $y \in W(X) \cap \{0,1\}^\N$ have positive lower asymptotic density. The set $N = \{n \;:\; y_n = 1\}$ is an independence set (which by assumption has positive lower asymptotic density): Suppose $(x^z)_{\N^\N \ni z \leq y}$ form a tree with branching structure $y$, i.e.,\ we have:
\[ x^z_{[0,i)} = x^{z'}_{[0,i)} \iff z_{[0,i)} = {z'}_{[0,i)}. \]
Supposing that $(m_i)_{i \in \N}$ enumerates $N$ in order, let $N_n = \{m_0,...,m_{n-1}\}$, and for $w \in \{0,1\}^n$ let $z^w \in \{0,1\}^\N$ be the characteristic sequence of $\{m_i \;:\; i < |w|, w_i = 1\}$. Clearly $z^w \leq y$ for all $w \in \{0,1\}^n$. The map:
\[ \phi : 2^n \to \{x^z|_{N_n} \;:\; z \leq b\} \]
defined by $\phi(w) = x^{z^w}|_{N_n}$ is injective, because $\phi(u0v)_{m_{|u|}} \neq \phi(u1v)_{m_{|u|}}$. Thus, this map is surjective, giving $X|_{N_n} = \{0,1\}^{N_n}$. By compactness, we have $X|_N = \{0,1\}^N$.
\end{proof}

\section{Proofs of supplementary claims}

In this section, I provide brief proofs for the claims made in the introduction.

\begin{proposition}
\label{prop:NoIndependenceEntropy}
There exists a subshift $X \subset \{0,1\}^\N$ with positive entropy, such that every topologically conjugate subshift $Y \cong X$ has zero independence entropy.
\end{proposition}

\begin{proof}
By a compactness argument, any subshift $X$ with positive independence entropy must contain two points which differ in only one position, i.e.,\ $\exists x, y \in X: x_0 \neq y_0 \wedge \forall i \neq 0: x_i = y_i$. Positive entropy does not imply the existence of a pair of points with finitely many differences, see \cite{Me19}, (Theorem~1.3), and clearly this property is preserved under conjugacy.
\end{proof}

\begin{proposition}
\label{prop:Countable}
For any function $f : \N \to \N$ satisfying $f(n) = o(a^n)$ for all $a > 1$, there exists a countable subshift $X \subset \{0,1\}^\N$ such that for all large enough $n$, the number of words of length $n$ in $X$ is at least $f(n)$.
\end{proposition}

\begin{proof}
The forbidden patterns of $X$ are described as follows: if a finite subword contains $n$ many $1$s, they must be pairwise separated by a distance of at least $m_n$. Here, $(m_n)_n$ is some nondecreasing sequence. The number of words of $X$ of length $n$ is at least $2^{\lfloor n/m_n \rfloor}$, and if $m_n$ grows slowly enough this stays above $f$. By the assumption on $f$ we can have $m_n \rightarrow \infty$, and it is easy to see that $X$ is then countable, since all its points have finite sum and finite subsets of $\N$ form a countable set.
\end{proof}

\begin{proposition}
\label{prop:NoSyndetic}
There exists a subshift $X \subset \{0,1\}^\N$ with positive entropy such that every point in $W(X)$ contains arbitrarily long subwords of the form $0^n$.
\end{proposition}

\begin{proof}
If $X$ itself is hereditary, then $W(W(X)) = W(X)$, so it suffices to find a hereditary $X$ with positive entropy such that every point $x \in X$ contains arbitrarily long subwords of the form $0^n$. For this, it clearly suffices to find a binary subshift $Y$ where every point contains the subword $0^n$ for arbitrarily large $n$, but some configuration contains $1$s with positive lower asymptotic density, as the smallest hereditary subshift containing $Y$ then has the claimed property: entropy comes from independently flipping the positive density sequence of $1$s, and taking the hereditary closure can clearly only add $0^n$-subwords to points. The existence of such $Y$ is folklore, see e.g.\ \cite{Sa20}, (Example~7.3), for an explicit construction.
\end{proof}

\bibliographystyle{plain}
\bibliography{../../../../../bib/bib}{}

\end{document}